\documentclass{amsart}

\usepackage{amsmath,amssymb}
\usepackage{amsfonts}
\usepackage{amsthm}
\usepackage{enumerate}
\usepackage{fancyhdr}
\usepackage{comment}

\makeatletter
\newcommand*\owedge{\mathpalette\@owedge\relax}
\newcommand*\@owedge[1]{%
  \mathbin{%
    \ooalign{%
      $#1\m@th\bigcirc$\cr
      \hidewidth$#1\m@th\wedge$\hidewidth\cr
    }%
  }%
}
\makeatother

\newcommand{\ol}{\overline}
\DeclareMathOperator{\vol}{vol}
\newcommand{\Rm}{\mathrm{Rm}}

\newtheorem{Thm}{Theorem}[section]

\newtheorem{Prop}[Thm]{Proposition}
\newtheorem{Cor}[Thm]{Corollary}

\newtheorem*{Rem}{Remark}
\numberwithin{equation}{section}
\newtheorem*{ThmA}{Theorem A}
\newtheorem*{ThmA'}{Theorem A'}
\newtheorem*{ThmB}{Theorem B}

\begin{document}

\newcommand{\Rn}{{\bold R}^n}
\newcommand{\R}{\bold R}

\title{Some aspects of Ricci flow on the $4$-sphere}
\author{Sun-Yung Alice Chang}
\address{Princeton University, Princeton NJ 08544 USA}
\email{chang@math.princeton.edu}
\author{Eric Chen}
\thanks{The second author was supported by an AMS--Simons travel grant.}
\address{University of California, Santa Barbara CA 93106-3080 USA}
\curraddr{University of California, Berkeley CA 94720-3840 USA}
\email{ecc@math.berkeley.edu}
\date{\today}

\begin{abstract}
In this paper, on 4-spheres equipped with Riemannian metrics we study some integral conformal invariants, the sign and size of which under Ricci flow characterize the standard 4-sphere. We obtain a conformal gap theorem, and for Yamabe metrics of positive scalar curvature with $L^2$ norm of the Weyl tensor of the metric suitably small, we establish the monotonic decay of the $L^p$ norm for certain $p>2$ of the reduced curvature tensor 
along the normalized Ricci flow, with the metric converging exponentially to the standard 4-sphere.
\end{abstract}

\maketitle

Throughout my career, I (the first author Alice Chang) have long admired
Vaughan Jones from a distance for the ingenuity of his works in
mathematics and for his open, earnest attitude during social encounters.
The conversations we shared in our brief meetings at UCLA, Berkeley,
at a summer school in New Zealand, and during the long bus ride from San
Pablo to Rio at ICM 2018 --- each has left a vivid memory. What a bright star
in our profession!

\section{Introduction}

We start by recalling some
earlier works of Chang--Gursky--Yang \cite{CGY03}.

Under what conditions on the curvature can we conclude that a smooth, closed 
Riemannian manifold is diffeomorphic (or homeomorphic) to the sphere?
A result which addresses this question is usually referred to as a {\it sphere theorem}, and the literature abounds with examples. 

An example of particular importance to us is the work of Margerin \cite {Ma}, in which he formulated a notion of ``weak curvature pinching."
To explain this we will need to establish some notation.  Given a Riemannian four-manifold $(M^4,g)$, let $Riem$ denote the curvature
tensor, $W$ the Weyl curvature tensor, $Ric$ the Ricci tensor, and $R$ the scalar curvature.  The usual decomposition of $Riem$ under
the action of $O(4)$ can be written   
\begin{align}
Riem = W + \frac{1}{2} E \odot g + {1\over 24} R g \odot g,\notag
\end{align}
where $E = Ric - {1 \over 4} Rg$ is the trace-free Ricci tensor and $\odot$ denotes the Kulkarni-Nomizu product.  If we let 
$Z = W + {1 \over 2} E \odot g$, then 
$$
Riem = Z + {1 \over 24} R g \odot g.
$$
Note that $(M^4,g)$ has constant curvature if, and only if, $Z \equiv 0$.  
We now define the scale-invariant ``weak pinching" quantity
\begin{align}
WP \equiv {|Z|^2 \over R^2} = {|W|^2 + 2|E|^2 \over R^2} ,
\end{align}
where $|Z|^2 = Z_{ijkl}Z^{ijkl}$ denotes the norm of $Z$ viewed as a $(0,4)$--tensor.   

Margerin's main result states that if $R > 0$ and $WP < {1 \over 6}$, then $M^4$ is diffeomorphic to either $S^4$ or $\Bbb R P^4$.
Moreover, this ``weak pinching" condition is {\it sharp}:  The symmetric spaces $(\Bbb C P^2, g_{FS})$ and $(S^3 \times S^1,
g_{prod.})$ both have $R > 0$ and $WP \equiv {1 \over 6}$.    

Margerin's proof relied on an important tool in the subject of sphere theorems, namely, Hamilton's Ricci flow.  In fact, previous
to his work, Huisken \cite {Hu} had used the Ricci flow to prove a similar pinching result, but with a slightly weaker constant.
In addition, Hamilton \cite {Ha} had used his flow to study four-manifolds with positive curvature operator.  As Margerin
points out in his introduction, there is no relation between weak pinching and positivity of the curvature operator; indeed,
weak pinching even allows for some negative sectional curvature.

One drawback to the sphere theorems described above is that they require one
to verify a {\it pointwise} condition on the curvature.  In contrast, consider
the (admittedly much simpler) case of surfaces.  For example, if the Gauss
curvature of the surface $(M^2,g)$ satisfies $\int K dA > 0$, then $M^2$ is
diffeomorphic to $S^2$ or $\Bbb R P^2$.  In addition to this topological
classification, the uniformization theorem implies that $(M^2,g)$ is conformal
to a surface of constant curvature, which is then covered isometrically by
$S^2$.  Therefore, in two dimensions one has a``sphere theorem" which only
requires one to check an {\it integral} condition on the curvature. 

In an earlier work of Chang--Gursky--Yang, they have reformulated the result of Margerin by
showing that  
the smooth four-sphere is also characterized by an 
integral curvature condition.  As we shall 
see, the condition has the additional properties of being {\it sharp} and
{\it conformally invariant}.\\
  
\begin{ThmA}[\cite{CGY03}]
Let $(M^4,g)$ be a smooth, closed manifold for which 
\begin{enumerate}[(i)]
\item the Yamabe invariant $Y(M^4,g) > 0$, and
\smallskip
\item the Weyl curvature satisfies
\begin{align}
\int_{M^4} |W|^2 dv < 16 \pi^2 \chi(M^4).\label{Acond}
\end{align}

\end{enumerate}
\smallskip\noindent
Then $M^4$ is diffeomorphic to either $S^4$ or $\Bbb R P^4$.\\
\end{ThmA}

\begin{Rem}
\ 
\begin{enumerate}[1.]
\item  Recall that the Yamabe invariant is defined by
\begin{align}
Y(M^4,g) = \inf_{\tilde g \in [g]} vol(\tilde g)^{-{1\over 2}}\int_{M^4} R_{\tilde g} dv_{\tilde g},
\end{align}
where $[g]$ denotes the conformal class of $g$.  Positivity of the Yamabe invariant implies that
$g$ is conformal to a metric of strictly positive scalar curvature \cite {KW}.  

\item  In the statement of Theorem A, the norm of the Weyl tensor is given by $|W|^2 = W_{ijkl}W^{ijkl}$;
i.e., the usual definition when $W$ is viewed as a section of $\otimes^4 T^{*}M^4$.
\end{enumerate}
\end{Rem}

By appealing to the Chern--Gauss--Bonnet formula, it is possible to replace \eqref{Acond} with a 
condition which does not involve the Euler characteristic.  Since
\begin{align}
8 \pi^2 \chi(M^4) = \int_{M^4} ( {1\over 4}|W|^2 - {1\over 2} |E|^2 + {1\over 24}R^2 )dv,   
\end{align}
one can also prove:\\

\begin{ThmA'}[\cite{CGY03}]
Let $(M^4,g)$ be a smooth, closed manifold for which 
\begin{enumerate}[(i)] 
\item the Yamabe invariant $Y(M^4,g) > 0$, and
\smallskip
\item the curvature satisfies
\begin{align}\label{GBintegrand}
\int_{M^4} ( -{1\over2} |E|^2 + {1\over 24}R^2 - {1\over4} |W|^2 ) dv > 0.
\end{align}
\end{enumerate}
\smallskip\noindent
Then $M^4$ is diffeomorphic to either $S^4$ or $\Bbb R P^4$. \\
\end{ThmA'}

Formulating the result of Theorem A in this manner allows us to explain the connection 
with the work of Margerin.  This connection relies on the works \cite {CGY1,CGY2}
in which they established the existence of solutions to 
a certain fully nonlinear equation in conformal geometry.  
Simply put, the results of \cite {CGY1} and \cite {CGY2} allow one to prove that under the hypotheses
of Theorem A', there is a conformal metric for which the integrand in (\ref{GBintegrand}) is
{\it pointwise} positive.  That is, through a conformal deformation of metric,
one is able to pass from positivity in an average sense to pointwise positivity.
Now, any metric for which the integrand in (\ref{GBintegrand}) is positive must satisfy
$$
|W|^2 + 2|E|^2 < {1\over 6}R^2,
$$
by just rearranging terms.
Note that this implies in particular that $R > 0$.  Dividing by $R^2$, we conclude that
$$
WP = {|W|^2 + 2|E|^2 \over R^2} < {1\over6}.
$$
The conclusion of the theorem thus follows from Margerin's work. \\

We remark that in \cite{CGY03}, they have also established that Theorem A is {\it sharp}. That is, one can
precisely characterize the case of equality 
as stated in Theorem B of that paper.

\begin{ThmB}[{\cite{CGY03}}]
Let $(M^4,g)$ be a smooth, closed manifold which is not 
diffeomorphic to either $S^4$ or $\mathbb{R}P^4$.  Assume in addition that 
\begin{enumerate}[(i)] 
\item the Yamabe invariant $Y(M^4,g) > 0$,
\smallskip
\item the Weyl curvature satisfies
\begin{align}
\int_{M^4} |W|^2 dv = 16 \pi^2 \chi(M^4).
\end{align}
\end{enumerate}
\smallskip\noindent
Then one of the following must be true:
\begin{enumerate}[1.]
\item $(M^4,g)$ is conformal to $\mathbb{C}P^2$ with the 
Fubini-Study metric $g_{FS}$, or
\smallskip
\item $(M^4,g)$ is conformal to a manifold which 
is isometrically covered by $S^3 \times S^1$
endowed with the product metric $g_{prod.}$.
\end{enumerate}
\end{ThmB}

The purpose of this note is twofold. First we will apply the method of proof of Theorem B in \cite{CGY03} to establish a conformal gap theorem, which we will state and establish as Theorem \ref{conformalgap} in Section \ref{confgapsec}. The second is to address a question of ``monotonicity" related to the proof of Theorem A along the Ricci flow, which we will formulate and discuss in detail in Section \ref{monoricsec}. In particular, we study the decay of the curvature quantities $|W|$, $|E|$, and $|R-\ol{R}|$, or equivalently the decay of the reduced curvature tensor $\left|Riem-\frac{1}{24}\ol{R} g\odot g\right|$, in the terminology of the abstract. We also remark that the monotonicity result of Section \ref{monoricsec} has been applied in some 
recent work of Chang--Prywes--Yang \cite{CPY21}to establish
the bi-Lipschitz equivalence of a class of metrics to the canonical metric on $S^4$  under some suitable curvature conditions.
 

\section{A conformal gap theorem}\label{confgapsec}

The proof of Theorem B in \cite{CGY03} relies on a vanishing result, 
in a sense which we will now explain.
Suppose $(M^4,g)$ satisfies the hypotheses of Theorem B. 
If there is another metric in a small neighborhood of 
$g$ for which the $L^2$-norm of the Weyl tensor is smaller; i.e., 
$$
\int_{M^4} |W|^2 dv < 16 \pi^2 \chi(M^4),
$$
then by Theorem A we would conclude that $M^4$ is diffeomorphic 
to either $S^4$ or ${\bold RP}^4$.
This, however, contradicts 
one of the assumptions of Theorem B.  Therefore, for every metric in
some neighborhood of $g$, 
$$
\int_{M^4} |W|^2 dv \geq 16 \pi^2 \chi(M^4).
$$
Consequently, $g$ is a critical point (actually, a local minimum) 
of the {\it Weyl functional} $g \mapsto \int |W|^2 dvol$. 
The gradient of this functional is called the {\it Bach
tensor}, and we will say that critical metrics are 
{\it Bach--flat}. 

We recall that in local coordinates, the Bach tensor is given by
\begin{align}
B_{ij} \, = \, 
\nabla^k 
\nabla^\ell 
W_{kij\ell} \, + \,
\frac{1}{2} \, 
R^{k \ell} \, W_{k i j \ell}.\notag
\end{align}

Thus, Bach-flatness $B = 0$ implies that the Weyl tensor lies in the kernel of a second order differential operator. The proof of Theorem B was established via some involved Lagrange-multiplier arguments in \cite{CGY03}.

One important fact we rely on in this note
is that both the Weyl tensor and the Bach tensor are
pointwise conformal invariants on 4-manifolds
in the sense that on $(M^4, g)$, if we change
the metric $g$ to $\hat g = e^{2w} g $ in the conformal class $[g]$ of $g$, then $W_{\hat g} = e^{-2w} W_g$, while $B_{\hat g} = e^{-2w} B_{g}$. Note that since $W$ is a 4-tensor and $B$ is a 2-tensor, this translates to fact that the $L^2$ norm of the 
Weyl tensor and the $L^1$ norm of the Bach tensor are conformal invariants in the integral sense, i.e. for all $g \in [g]$,
\begin{align} 
\int_{M^4} |W_{\hat g}|_{\hat g}^{2} dv_{\hat g} = \int_{M^4} |W_g|_g^2 dv_g\notag
\end{align}
and
\begin{align}
\int_{M^4} |B_{\hat g}|_{\hat g} dv_{\hat g} = \int_{M^4} |B_g|_g dv_g,\notag
\end{align}
where $|W_g|_g $ denotes the g-norm of the Weyl tensor in $g$ metric; which, to simplify the notations we will henceforth denote by $|W|_g$ or simply $|W|$.
 
We now state a conformal gap theorem which is the main result of this section. 

\begin{Thm} \label{conformalgap}
On $(S^4, g)$, if $Y(M^4, g)> 0$ and $g$ is Bach-flat with 
\begin{align} \label{gap}
\int_{S^4} |W|_g^2 dv_g \leq 32 \pi^2,
\end{align}
then $W_g \equiv 0$; hence $g \in [g_c]$.
\end{Thm}

We remark that an important class of metrics
which are Bach-flat is the class of metrics which are conformal to some Einstein metric. One way to see this is through a well-known relationship between the Weyl tensor and the Ricci tensor (\cite{De}, also \cite{CGY03}, p.124). 
which can be derived as a consequence of the Bianchi identity:
\begin{align}
\nabla_k W_{ijlk} : = (\delta W) _{ijl} = \frac{1}{2}(dA)_{ijl},\notag
\end{align}
where
\begin{align} 
(dA)_{ijl} = \nabla_iA_{jl} \,- \, \nabla_j A_{il},\notag
\end{align}
and where
$A_{ij} = Ric_{ij} - \frac{1}{6} R g_{ij}$
denotes the Schouten tensor. Thus in particular 
an Einstein metric is Bach-flat, and as the later condition is pointwisely conformally invariant, so is a metric conformal to an Einstein metric. 
Therefore as a corollary to Theorem \ref{conformalgap}, we have

\begin{Cor}
The only Einstein metric $g$ on $S^4$ with
positive scalar curvature and
\begin{align}
\int_{S^4} |W|_g^2 dv_g \leq 32 \pi^2
\end{align}
is conformally equivalent to the canonical metric $g_c$ on $S^4$.
\end{Cor}

\begin{Rem}
 
We would like to point out it remains an open question whether on the 4-sphere there exists an Einstein metric with positive scalar curvature other than the canonical metric $g_c$. 
It turns out the ``gap" provided in the corollary above in terms of the size of Weyl curvature for Einstein metric is not the best one known. In an earlier work of Gursky \cite{Gu00}, he has studied the class of metrics whose self-dual part of Weyl tensor is harmonic (i.e. $\delta W^+ = 0$, where $\delta$ denotes the divergence), which again includes the Einstein metrics as special cases, and the conclusion in Theorem C in that paper presents a gap bigger than ours. More precisely, Gursky's result states that the only Einstein metric $g$ on $S^4$ with
positive scalar curvature and
$$\int_{S^4} |W|_g^2 dv_g \leq \frac{128}{3} \pi^2   $$
is conformally equivalent to the canonical metric $g_c$ on $S^4$.

 \end{Rem}

We now present a proof of Theorem 2.1 which largely follows from the proof of Theorem B in
\cite{CGY03}. We start by stating two integral 
equalities for Bach-flat metrics on 4-manifolds. 

\begin{Prop}[See {\cite[Lemma 5.4]{CGY1}}] \label{prop1}
If $(M^4 , g)$ is Bach-flat, then
\begin{align} \label{identity1}
0 = \int\limits_{M^4} \left\{ 3 \, \left( |\nabla E|^2 - {1\over 12}
    |\nabla R|^2 \right) + 6tr E^3 + R|E|^2 -
      6W_{ijk\ell} E_{ik} E_{j \ell}\right\} dv \,,
\end{align}
where $tr E^3 = E_{ij} E_{ik} E_{jk}$.
\end{Prop}

The second identity \eqref{identity2} below is a Weitzenb\"ock formula for Bach-flat metrics, derived 
in \cite{CGY03} using the Singer-Thorpe \cite{ST} decomposition of the curvature tensor on 4-manifolds viewed as curvature operators acting on 2 forms; the derivation is also motivated by the decomposition of the curvature tensor used in the proof of the main result in the paper of Margerin \cite{Ma}; we remark that establishing the
identity (\ref{identity2}) is actually one of the (two) major steps in the proof of Theorem B in \cite{CGY03}.

\begin{Prop}[{\cite[Proposition 3.3]{CGY03}}]\label{prop2}
If $(M^4, g)$ is Bach-flat, then
\begin{align} \label{identity2}
0 \,= \, \int_{M^4} |\nabla W |^2 - 
   \left\{ 72 \det W^+ + 72 \det W^- -
  {1\over 2} R |W|^2 + 2 W _{ij\kappa\ell}E_{i\kappa} E_{j\ell} \right\} dv.
\end{align}
\end{Prop}
\vskip .2in

We remark that in the identity (\ref{identity2}) we have viewed the Weyl tensor as a
curvature operator acting on 2-forms, and
$W^+$ and $W^-$ respectively as the self-dual and anti-self-dual parts of the Weyl operator.
\\ 

For the rest of this section we continue with the proof of Theorem \ref{conformalgap}. We first introduce a notion of elementary symmetric polynomials
$\sigma_\kappa : \mathbb{R}^n \rightarrow \mathbb{R}^n$,
$$
\sigma_\kappa ( \lambda_1 , \ldots , \lambda_n) =
   \sum\limits_{i_1 < \cdots < i_k}
     \lambda_{i_1} \cdots \lambda_{i_\kappa} \;.
$$
For a section ${\it{ S}} $ of End$(TM^4)$ --- or, equivalently, a
section of $T^\star M^4 \otimes TM^4$ ---
the notation $\sigma_\kappa (\it {S})$ means $\sigma_\kappa$ applied to the 
eigenvalues of $\it {S}$. In particular, given a section of 
the bundle of symmetric two--tensors such as $A$, 
by ``raising an index'' we can canonically associate a section
$g^{-1}A$ of  End$(TM^4)$.  At each point of $M^4$,
$g^{-1}A$ has 4 real eigenvalues; thus $\sigma_\kappa(g^{-1} A)$
is a smooth function on $M^4$.  To simplify notation, we
denote $\sigma_\kappa (A) = \sigma_\kappa (g^{-1}A)$.

Using this notion, we can rewrite
the Chern--Gauss--Bonnet formula as
\begin{align} \label{CGB2}
8 \pi ^2 \chi (M^4) =
   \int {1\over 4 }|W|^2 dvol +
     \int \sigma_2 (A)  dv \;.
\end{align}
where
\begin{align} \label{s2a}
\sigma_2 (A) = \frac{1}{24} R^2 - \frac{1}{2} |E|^2.
\end{align}

Returning to the consider $(S^4,g)$ satisfying the
assumptions of Theorem \ref{conformalgap}, a  key step in the proof is to pass the information in the integral condition (\ref{gap}) to a solution of a PDE satisfied by a metric $\hat g$ in 
the conformal class $[g]$.

\begin{Prop}[{\cite[Theorem 1.1 and Lemma 3.1]{CGY03}}]\label{chosenpde}
Suppose $(S^4, g)$ is smooth and satisfies the same conditions as in the statement of Theorem \ref{conformalgap}. Then there is a smooth metric
$\hat g \in [g]$, so that w.r.t the metric $\hat g$, 
\begin{align} \label{pde}
\sigma_2 (A)  - {1 \over 4} |W|^2 = \lambda\,,
\end{align}
where $\lambda \geq 0$ is a nonnegative constant.
Also as a consequence we have 
\begin{align} \label{inqpde}
{3\over 2 }  |\nabla W|^2 + 3 ( |\nabla E|^2 -
   {1\over 12} | \nabla R|^2) \geq 0 \,.
\end{align}
\end{Prop}

From now on for the rest of this section we work on the metric $\hat g$ chosen in Proposition \ref{chosenpde} above.
\vskip .1in

For such a metric, by combining the identities (\ref{identity1}) and (\ref{identity2}), and applying the inequality 
(\ref{inqpde}), we obtain that 
\begin{align} 
0& = \int_{S^4} \left\{ {3\over 2} 
   |\nabla W|^2 + 3 \left(|\nabla E|^2 - {1\over 12} 
       |\nabla R|^2 \right) \right.\\
&\qquad\qquad + 6trE^3 + R|E|^2 - 9 
      W_{ij\kappa \ell} E_{i\kappa}E_{j\ell}\notag  \\
&\qquad \qquad \left. - 108 \det W^+ - 108 \det W^- +
   {3\over 4} R | W |^2 \right\} dv\notag \\
&\geq \int_{S^4} 6trE^3 + R|E|^2 - 9 
      W_{ij\kappa \ell} E_{i\kappa}E_{j\ell} \notag \\ 
&\qquad \qquad  - 108 \det W^+ - 108  \det W^- +
   {3\over 4} R | W |^2\ dv  \,.\notag
\end{align} 

We now denote and analyze the integrand,  
\begin{align}\label{integrand}
G : = 6trE^3 + R|E|^2 - 9 
      W_{ij\kappa \ell} E_{i\kappa}E_{j\ell}   - 108 \det W^+ - 108  \det W^- +
   {3\over 4} R | W |^2  \,.
\end{align}

Let 
\begin{align}
S_+^4 & = \{ x \in S^4 : |W_g| > 0 \} ,\notag \\
S_0^4 &= \{ x \in S^4 : | W_g| = 0 \} \,.\notag
\end{align}

Observe that at points $p \in S_0^4$, we have
\begin{align} \label{estimateG}
G = 6trE^3 + R|E|^2 \geq - 2 \sqrt{3} \,|E|^2 + R|E|^2
 = |E|^2 \big(R - 2 \sqrt{3} \, |E| \big) \;
\end{align}

Since $\sigma_2(A) \geq 0$ on $S_0^4$, 
i.e. 
$0 \leq  - {\frac{1}{2}} |E|^2 + {\frac{1}{ 24}} R^2 $, which implies 
$ R^2 \geq  12 |E|^2$, then
\begin{align} \label{estimateG2}
R \geq 2 \sqrt{3}\,|E| \,.
\end{align}

Combining the inequalities in (\ref{estimateG}) and (\ref{estimateG2}), we see that the integrand $G$ in \eqref{integrand} is non-negative on the set $S_0^4$.\\

The next step is the delicate part of the proof. For a point $p\in S_+^4$, 
we recall $\hat g $ satisfies 
$$
\sigma_2(A) = {\frac{1}{4}} | W|^2 + \lambda  = \Vert W \Vert^2 + \lambda
  = \Vert W^+\Vert^2 + \vert W^- \Vert ^2 + \lambda.
$$
for some constant $\lambda \geq 0$, where where $\|\cdot\|$ denotes the norm of $W$ when viewed as an endomorphism of $\Lambda^2 T^* M$, related to the norm of $W$ as a tensor by $\|W\|=\frac{1}{4}|W|^2$.
Therefore,
\begin{align}
R = \sqrt{12 |E|^2 + 24 \Vert W^+\Vert ^2
     + 24\Vert W^-\Vert ^2 + \lambda }.\notag
\end{align}

Now suppose we let 
\begin{align}
\tilde R = \sqrt{12 |E|^2 + 24 \Vert W^+\Vert ^2 + 24\Vert W^-\Vert ^2} \notag
\end{align}
and denote by $\tilde G$ the same expression that defines 
$G$ but with $R$ substituted by $\tilde R$.
A crucial step (Proposition 4.4 in the proof of Theorem B in \cite{CGY03}) is to show that one can estimate $\tilde G$ at points
$p$ on the set $S_+^4$ and show that the integrand $\tilde G$ 
is non-negative at such points. We now split the rest of the proof into two cases.\\ 

\noindent{Case 1: when $\lambda >0$.}

Note that the above argument implies in this case when $ \lambda $ is positive in the equation (\ref{pde}), which
corresponds to the case when  
$$ \int_{S^4} |W|^2 dv < 32 \pi^2, $$
that at points $p \in S_+^4$, we have
$G (p) \geq {\tilde G} + \frac{1}{4} (R - \tilde R)|W|^2 > 0.$ This coupled with the facts which have 
already established earlier, namely that $G$ is non-negative in the set $S_0^4$ and by (2.16) that the integration of $G$ over the whole $S^4$ is non-positive, implies
that the set $S_+^4$ is a set of measure zero. This coupled with the fact that the metric
$\hat g$ is a smooth metric implies that actually the Weyl curvature of the metric $\hat g$ vanishes identically; hence the statement of the 
Theorem \ref{conformalgap} is valid in this case.\\

\noindent{Case 2, when $\lambda = 0$.}

We now address the case when 
$$ \int_{S^4} |W|^2 dv = 32 \pi^2, $$
or equivalently when $\lambda =0.$ We recall 
another result
in \cite{CGY03}, which was established based
on Proposition 4.1 in that article. We remark that the statement of the Proposition holds for general compact closed 
4-manifolds $M$ but later we will apply it for the special case when $M$ is $S^4$. \\

\begin{Prop}[{\cite[Proposition 4.1]{CGY03}}] \label{prop4.1}
Suppose $\hat g$ is a metric in $[g]$, 
$\hat g = e^{2w}g$ on $M^4$ with $w \in C^\infty (M_+^4) \cap C^{1,1}(M^4)$
which satisfies 
\begin{align} \label{limitpde}
\sigma_2(A) \, = \, \frac{1}{4} |W|^2 \, \text{on } \, M_+^4, 
\end{align}
with 
\begin{align} \label{cinq}
\int_{M^4} G dv\leq 0.
\end{align}
Then we have
\begin{enumerate}[(i)]
\item $\hat g  \in C^\infty (M^4)$; \newline
\item $\hat g$ is Einstein; \newline
\item\label{41pt3} either $W^+ \equiv 0$ or $W^- \equiv 0$ on $M^4$.
\end{enumerate}
\end{Prop}

We now apply Proposition \ref{prop4.1} to finish the proof of Theorem \ref{conformalgap}
by asserting that when the manifold $M^4$ is topologically $S^4$, by the Hirzebruch signature formula we have
$$
0 = 12 \pi^2 \tau (S^4) = \int_{S^4} |W^+|^2 dv - \int_{S^4} |W^-|^2 dv.
$$
Thus condition \eqref{41pt3} in Proposition \ref{prop4.1}
implies that both $W^+$ and $W^-$ vanish identically, so the metric g is conformally equivalent to the canonical metric $g_c$.

We have thus finished the proof of Theorem \ref{conformalgap}.
 


\section{Monotonicity along the Ricci flow}\label{monoricsec}

Recall that Theorems A and A' can be viewed as integral versions of Margerin's pointwise pinching sphere theorem \cite{Ma}. This latter result, as described in the Introduction, is obtained by studying the convergence of the normalized Ricci flow when starting from a metric satisfying a pointwise ``weak pinching'' property. In contrast, Theorems A and A' are established by a combination of conformal geometry techniques and Margerin's result, and therefore do not imply that the normalized Ricci flow must converge when starting from metrics satisfying integral curvature pinching as in condition \eqref{Acond}. A natural question then is whether one can in fact establish that this flow must converge when starting from metrics satisfying such integral pinching via monotonicity of some integral form of the curvature along the Ricci flow --- if so this would give a direct proof of Theorems A and A', without a need to quote Margerin's work as was done in \cite{CGY03}. 



A most natural integral quantity to try is the
$L^2$ norm of the Weyl tensor. Thus the question we ask is: Under the assumption that
we start with a ``nice" metric on $S^4$, e.g. some Yamabe metric with constant positive scalar curvature, and the $L^2$ norm of Weyl tensor
less than $32 \pi^2$ as in the statement of
assumption of Theorem A, does the $L^2 $ norm of the Weyl tensor monotonically decrease to 
zero along the (normalized) Ricci flow and does the metric converge to the standard metric? It turns out we can only partially answer this
question in the sense that we can achieve this
under the assumption the $L^2$ norm of
the Weyl tensor is sufficiently
small (of an estimable size):

\begin{Thm}\label{3main}
There exists $\epsilon_0>0$ such that if $(S^4,g_0)$ is a Yamabe metric with $Y(S^4,[g_0])>0$ and $\int |W|^2dv<\epsilon_0$, then 
\begin{align}
\int (\|W\|^p+\frac{1}{2}|E|^p)\,\, dv_{g(t)}\quad\text{and}\quad\int (\|W\|^p+\frac{1}{2}|E|^p+\frac{1}{10^6}|R-\ol{R}|^p )\,\, dv_{g(t)}\label{int3}
\end{align}
are monotonically and exponentially decreasing under the normalized Ricci flow $(S^4,g(t))$ starting from $(S^4,g_0)$, for any $p\in\left[2,\frac{8}{3}\right]$.

Moreover, $\|\|W\|+|E|+|R-\ol{R}|\|_\infty\rightarrow 0$ as $t\rightarrow\infty$ and the flow converges exponentially to $(S^4,g_c)$.
\end{Thm}
\begin{Rem}
Above, $\ol{R}$ denotes the average scalar curvature over $S^4$, and 
where $\|\cdot\|$ denotes the norm of $W$ when viewed as an endomorphism of $\Lambda^2 T^* M$, related to the norm of $W$ as a tensor by $\|W\|=\frac{1}{4}|W|^2$. From the arguments used to prove Theorem \ref{3main} one could for instance take $\epsilon_0=\frac{1}{2000}\pi^2$.
\end{Rem}

The convergence part of the above result is originally due to Gursky \cite{Gursky}, but unlike Gursky's and other previous studies of integral pinching under the Ricci flow \cite{HV,CWY20}, Theorem \ref{3main} does not rely on quoting the pointwise pinching work of Huisken, Margerin, and Nishikawa to obtain convergence of the flow, and is instead proven directly from the monotonicity of the integral quantities \eqref{int3}. We also remark that the requirement that we start with a ``nice'' metric is natural in combination with the restriction on the size $\int |W|^2\ dv$, in view of the neckpinch examples of Angenent--Knopf \cite{AK} which show that even when $\int |W|^2\ dv=0$ as for metrics in $[g_c]$, some metrics in this conformal class do not converge under the Ricci flow to $(S^4,g_c)$.

In the rest of this section we collect and justify the necessary estimates to prove Theorem \ref{3main}.

\subsection{Preliminaries}


We say that $(S^4,g(t))$ is a Ricci flow starting from some intial metric $g(0)=g_0$ on $S^4$ if $g(t)$ satisfies the equation $\frac{\partial}{\partial t} g=-2\text{Ric}$. We will also work with the normalized Ricci flow equation, which in dimension four is given by $\frac{\partial}{\partial t} g=-2\text{Ric}+\frac{1}{2}\overline{R} g$. Finally we recall two additional facts from conformal geometry:

First, on $S^4$ we rewrite the Chern--Gauss-Bonnet formula as
\begin{align}
16\pi^2=\int\|W\|^2-\frac{1}{2}|E|^2+\frac{1}{24}R^2\ dv_g.\label{CGB}
\end{align}

Second, we have the following conformally invariant Sobolev inequality:
\begin{align}
Y(S^4,[g]) \|u\|_4^2\leq 6\int\|\nabla u\|^2\ dv_g+\int R u^2\ dv_g.\label{ysob}
\end{align}
Note in particular that the Yamabe constant $Y(S^4,[g])$ which appears above depends only on the conformal class of the metric $g$.




From now on, all integration will be assumed to be over $S^4$.  We start with some useful consequences of \eqref{CGB}.

\begin{Prop}\label{cons1}
Let $(S^4,[g])$ have positive Yamabe constant. Then if $g_Y\in[g]$ is a Yamabe metric,
\begin{enumerate}[(a)]
\item $\frac{1}{24}\int R_{g_Y}^2\ dv_{g_Y}\leq 16\pi^2$,\label{YR}
\item $\frac{1}{2}\int|E|^2_{g_Y}\ dv_{g_Y}\leq\int\|W\|^2_g\ dv_g$,\label{EW}
\item $Y(S^4,[g])\geq\left[24\left(16\pi^2-\int\|W\|^2\ dv_g\right)\right]^{\frac{1}{2}}$.\label{Ybound}
\end{enumerate}
\end{Prop}
\begin{proof}
For \eqref{YR}, first note that $\int R_{g_Y}^2\ dv_{g_Y}=Y(S^4,g_Y)^2$. By Aubin's work on the Yamabe problem we know that $Y(S^4,g_Y)$ is always bounded from above by the Yamabe constant of the conformal class of the round $(S^4,g_c)$, hence
\begin{align}
\frac{1}{24}\int R_{g_Y}^2\ dv_{g_Y}\leq\frac{1}{24}Y(S^4,g_c)^2\leq 16\pi^2.\notag
\end{align}
Then \eqref{EW} follows by combining \eqref{YR} with \eqref{CGB} and noting that $\|W\|^2_g\ dv_g$ is a pointwise conformal invariant.

Finally \eqref{Ybound} follows from combining $\int R_{g_Y}^2\ dv_{g_Y}=Y(S^4,g_Y)^2$, the conformal invariance of $\|W\|^2_g\ dv_g$, and \eqref{CGB}.
\end{proof}

We now point out a key fact which we shall use throughout this note.  
Since $\int\|W\|^2_g\ dv_g$ is a conformal invariant, then by \eqref{CGB} we see that
\begin{align}
\frac{1}{24}\int R_g^2\ dv_g-\frac{1}{2}\int|E|^2_g\ dv_g=16\pi^2-\int\|W\|^2_g\ dv_g
\end{align}
is also a conformal invariant. This leads us to the following comparison between Yamabe metrics and other metrics in the same conformal class.

\begin{Prop}\label{R2boundp}
For any conformal class $(S^4,[g])$ with positive Yamabe constant, we have for all $g\in[g]$ that
\begin{align}
\frac{1}{24}\int(R_g-\ol{R_g})^2\ dv_g&=\frac{1}{24}\left(\int R_g^2\ dv_g-\ol{R_g}^2\vol(g)\right)\label{RE}
\\
&\leq\frac{1}{24}\left(\int R_g^2\ dv_g-R_{g_Y}^2\vol(g_Y)\right)\notag
\\
&=\frac{1}{2}\left(\int |E|^2_g\ dv_g-\int|E|^2_{g_Y}\ dv_{g_Y}\right).\notag
\end{align}
Above, $\ol{R_g}$ denotes the average of scalar curvature $\frac{\int R_g\ dv_g}{\vol(g)}$.
\end{Prop}
\begin{proof}
The second line follows from the fact that
\begin{align}
R_{g_Y}\vol(g_Y)^{1/2}=Y(S^4,[g])\leq \frac{\int R_g\ dv_g}{\vol(g)^{1/2}},\notag
\end{align}
so that $R_{g_Y}^2\vol(g_Y)\leq\ol{R_g}^2\vol(g)$; the last inequality follows from \eqref{RE}.
\end{proof}

Finally we record the evolution equations for the Weyl, Einstein, and scalar curvatures under the Ricci flow on a compact $4$-manifold. These may be found for instance in \cite[Lemma 4.5]{CGZ}, besides the evolution of scalar curvature which is well known.
\begin{Prop}\label{evoform}
Under the Ricci flow on a compact $4$-manifold we have
\begin{align}
&\frac{\partial}{\partial t}\|W\|^2=\Delta\|W\|^2-2\|\nabla W\|^2+36\det W+WEE.\label{Wevo}
\\
&\frac{\partial}{\partial t}|E|^2=\Delta |E|^2-2|\nabla E|^2+4WEE-4\operatorname{tr}E^3+\frac{2}{3}R|E|^2,\label{Eevo}
\\
&\frac{\partial}{\partial t}(R-\ol{R})^2=\Delta(R-\ol{R})^2-2|\nabla R|^2\label{Revo}
\\
&\qquad\qquad\qquad+(R-\ol{R})[4|E|^2-4\ol{|E|^2}+R^2-2\ol{R}^2+\ol{R^2}]\notag
\end{align}
\end{Prop}

\subsection{Control of the Yamabe constant under the Ricci flow}

We want to control the Yamabe constant in order to estimate the evolution of curvature quantities under the Ricci flow. By Proposition \ref{cons1} \eqref{Ybound}, we see that along the flow $Y(S^4,[g(t)])$ has a lower bound in terms of an upper bound on $\int\|W\|^2_{g(t)}\ dv_{g(t)}$. So we can achieve the desired control by showing that
\begin{align}
F_2(t)=\int\left(\|W\|^2_{g(t)}+\frac{1}{2}|E|^2_{g(t)}\right) \ dv_{g(t)}
\end{align}
is decreasing under the flow, provided that $\int\|W\|^2_{g(0)}\ dv_{g(0)}<\epsilon_0$, for $\epsilon_0>0$ sufficiently small.

\begin{Prop}\label{EWmono}
If $(S^4,g(t))$ is a Ricci flow starting from a positive Yamabe metric $(S^4,g_0)$ with $\int\|W\|^2_{g_0}\ dv_{g_0}<\epsilon_0=\frac{8}{25}\pi^2$, then $F_2(t)$ is monotonically decreasing under the flow.
\end{Prop}
\begin{proof}
From \eqref{Wevo} and \eqref{Eevo} we have (as in \cite[Corollary 4.7]{CGZ}),
\begin{align}
\frac{d}{dt}\int \|W\|^2\ dv_t&\leq -\int \left(2|\nabla \|W\||^2+ R\|W\|^2\right)\ dv_t\label{W2evo}
\\
&\quad+2\sqrt{6}\|W\|_2\|W\|_4^2+\frac{\sqrt{6}}{3}\|W\|_2\|E\|_4^2,\notag
\\
\frac{d}{dt}\int \frac{1}{2} |E|^2\ dv_t&\leq -\int \left(|\nabla E|^2+\frac{1}{6} R|E|^2\right)\ dv_t\label{E2evo}
\\
&\quad+2\|W\|_2\|E\|_4^2+2\|E\|_2\|E\|_4^2.\notag
\end{align}
Therefore by using the conformally invariant Sobolev inequality \eqref{ysob}, with $Y$ denoting $Y(S^4,[g(t)])$ below we have
\begin{align}
\frac{d}{dt}F_2(t)&\leq \left(-\frac{Y}{3}+2\sqrt{6}\|W\|_2\right)\|W\|_4^2\notag
\\
&\quad+\left(-\frac{Y}{6}+(2+\frac{\sqrt{6}}{3})\|W\|_2+2\|E\|_2\right)\|E\|_4^2.\notag
\end{align}
From this, by using Proposition \ref{cons1} \eqref{Ybound} we see that the coefficients of both $\|W\|_4^2$ and $\|E\|_4^2$ on the right will be negative if
\begin{align}
-\frac{1}{6}\left[24(16\pi^2-F_2)\right]^{\frac{1}{2}}+4\sqrt{ F_2}<0,\notag
\end{align}
or equivalently if $F_2<\frac{16}{25}\pi^2$. So the result on the smallness of $\int\|W\|^2_{g_0}\ dv_{g_0}$ follows by Proposition \ref{cons1} \eqref{EW}.
\end{proof}

We now consider the evolution of curvature under normalized Ricci flow. As in \cite{Hamilton}, if $(S^4,g(t))$ satisfies the Ricci flow equation, then by setting $\psi(t)=\text{Vol}(g(t))^{-1/2}$ and rescaling
\begin{align}
\tilde{g}(t)=\psi g(t),\quad \tilde{t}=\int_0^t\psi(s)\ ds,\notag
\end{align}
we obtain $(S^4,\tilde{g}(\tilde{t}))$ which solves the normalized Ricci flow equation
\begin{align}
\frac{\partial}{\partial\tilde{t}}\tilde{g}=-2\tilde{\mathrm{Ric}}+\frac{1}{2}\ol{\tilde{R}}\tilde{g}.\notag
\end{align}
In particular $\operatorname{Vol}(S^4,\tilde{g}(\tilde{t}))\equiv 1$ for all $\tilde{t}$ for which the flow is defined.  Note also that $F_2$ is scaling-invariant, so this quantity is the same whether at $g(t)$ or $\tilde{g}(\tilde{t})$. We can then reformulate Proposition \ref{EWmono} in terms of the normalized Ricci flow.

\begin{Cor}\label{EWmonolem}
If $(S^4,g(t))$ is a normalized Ricci flow starting from a unit volume positive Yamabe metric $(S^4,g_0)$ with $\int\|W\|_{g_0}^2\ dv_{g_0}<\epsilon_0=\frac{8}{25}\pi^2$, then $F_2(t)$ is exponentially decreasing under the flow.
\end{Cor}
\begin{proof}
From the proof of Proposition \ref{EWmono}, we know that under a Ricci flow starting from a positive Yamabe metric $(S^4,g_0)$ with $\int\|W\|_{g_0}^2\ dv_{g_0}<\frac{8}{25}\pi^2$ we have
\begin{align}
\frac{d}{dt}F_2(t)\leq-CF_2(t)\vol(g(t))^{-1/2},\notag
\end{align}
where the constant $C$ depends on a positive lower bound for $\frac{8}{25}\pi^2-\int\|W\|_{g_0}^2\ dv_{g_0}$. Passing to the normalized Ricci flow $(S^4,\tilde{g}(\tilde{t})$, we find from direct computation that
\begin{align}
\frac{d}{d\tilde{t}}\tilde{F_2}\leq -C\tilde{F_2}(\tilde{t}),\notag
\end{align}
with $C>0$ the same constant as before, which gives us the conclusion.
\end{proof}

Now that we have control of both $\|W\|_2$ and $\|E\|_2$, we also obtain a uniform Sobolev inequality along the normalized Ricci flow. However, \eqref{ysob} itself is enough for the purpose of obtaining monotonicity of integral curvature quantities along the flow. The result below is only needed afterwards when applying Moser's weak maximum principle to pass to $L^\infty$ control.

\begin{Prop}\label{usob}
If $(S^4,g(t))$ is a normalized Ricci flow starting from a unit volume positive Yamabe metric $(S^4,g_0)$ with $\int\|W\|_{g_0}^2\ dv_{g_0}<\epsilon_0=\frac{8}{25}\pi^2$, then there exists $C>0$ such that for every $u\in W^{1,2}(S^4,g(t))$,
\begin{align}
\left(\int |u|^4\ dv_t\right)^{1/2}\leq C\left(\int|\nabla u|^2+u^2\ dv_t\right).
\end{align}
\end{Prop}
\begin{proof}
We start by rewriting the conformally invariant Sobolev inequality along the Ricci flow \eqref{ysob} as
\begin{align}
Y\|u\|_4^2\leq 6\int|\nabla u|^2+ (R-\overline{R}) u^2+\overline{R} u^2\ dv_t.\label{usobstep}
\end{align}
In particular this inequality is invariant with respect to scaling of $g(t)$, so it takes the same form on the corresponding time-slice of the normalized Ricci flow. So for the rest of this proof we will work exclusively with the normalized flow. 

By Corollary \ref{EWmonolem} and Proposition \ref{cons1} \eqref{Ybound} we have the lower bound $Y(S^4,g(t))\geq\left[24(16\pi^2-F_2(0))\right]^{\frac{1}{2}}$ along the flow, while by Proposition \ref{R2boundp} we have $\int (R-\ol{R})^2\ dv_t\leq 24F_2(0)$. Moreover $\ol{R}$ is uniformly bounded along the normalized flow, since
\begin{align}
\frac{1}{24}\ol{R}^2\leq\frac{1}{24}\int R^2\ dv_{\tilde{g}}\leq 16\pi^2+F_2(0).\notag
\end{align}
Thus we can conclude by absorbing the $\int (R-\overline{R}) u^2\ dv_t$ term on the left-hand side of \eqref{usobstep} as long as
\begin{align}
\left[24(16\pi^2-F_2(0))\right]^{\frac{1}{2}}-\left[24 F_2(0)\right]^{\frac{1}{2}}>0,\notag
\end{align}
which occurs if $F_2(0)<8\pi^2$. And this is certainly true since we assumed $F_2(0)<\frac{16}{25}\pi^2$.
\end{proof}

\subsection{Monotonicity of $W$, $E$, and $(R-\ol{R})$ under the normalized Ricci flow}

Now we will establish the monotonicity of the curvature quantities
\begin{align}
G_p(t)=\int \|W\|^p+\frac{1}{2}|E|^p+a |R-\ol{R}|^p\ dv_t
\end{align}
along the normalized Ricci flow, for $p=2+\delta$ and $\delta\in[0,\frac{2}{3})$, with $a>0$ a small constant to be established later. We first present the argument in the special case $\delta=0$ to illustrate our strategy before giving the general argument.

\subsection{The case $\delta=0$}

We start by collecting the evolution equations for $W$, $E$, and $|R-\ol{R}|$ that we will need.  For $W$ we retain \eqref{W2evo}, and next we rewrite \eqref{E2evo} (after some work and applying Proposition \ref{R2boundp}) as
\begin{align}
\frac{d}{dt}\int\frac{1}{2}|E|^2\ dv_t&\leq -\int\left(|\nabla E|^2+\frac{1}{6}R|E|^2\right)\ dv_t+2\|W\|_2\|E\|_4^2+2\|E\|_2\|E\|_4^2\label{Edelta2}
\\
&\leq  -\frac{Y}{12}\|E\|_4^2+\left(2\|W\|_2+(2+\sqrt{3})\|E\|_2\right)\|E\|_4^2\notag
\\
&\quad-\frac{1}{12}\ol{R}\int|E|^2\ dv_t.\notag
\end{align}
Note that the last term is negative because we assumed that $R_{g_0}>0$, which implies that $\ol{R}>0$ along the flow. Thus the strategy of our proof is to keep this negative term and use it to cancel the positive term $\ol{R}\int(R-\ol{R})^2\ dv_t$ which will arise from the result of the computation of $\frac{d}{dt}\int(R-\ol{R})^2\ dv_t$ below, again obtained after some manipulation with the help of the inequality in Proposition \ref{R2boundp}.

We rewrite $\frac{d}{dt}\int (R-\ol{R})^2\ dv_t$ as
\begin{align}
\frac{d}{dt}\int (R-\ol{R})^2\ dv_t&=-2\int|\nabla R|^2\ dv_t+4\int(R-\ol{R})|E|^2\ dv_t+\int\ol{R}(R-\ol{R})^2\ dv_t
\\
&\leq -\frac{Y}{3}\|R-\ol{R}\|^2_4+\frac{1}{3}\|R-\ol{R}\|_2\|R-\ol{R}\|_4^2+16\ol{R}\int|E|^2\ dv_t\notag
\\
&\quad+4\|R-\ol{R}\|_2\|E\|_4^2.\notag
\end{align}
Thus for any $a>0$,
\begin{align}
\frac{d}{dt}\int a(R-\ol{R})^2\ dv_t&\leq -\frac{Y}{3} a\|R-\ol{R}\|_4^2+\frac{2}{\sqrt{3}}a\|E\|_2\|R-\ol{R}\|_4^2+16a\ol{R}\int|E|^2\ dv_t\label{Rdelta2}
\\
&\quad+4a\|R-\ol{R}\|_2\|E\|_4^2.\notag
\end{align}
Adding together \eqref{W2evo}, \eqref{Edelta2}, and \eqref{Rdelta2}, we find that if $a\leq\frac{1}{192}$, then
\begin{align}
&\frac{d}{dt}\int \|W\|^2+\frac{1}{2}|E|^2+a(R-\ol{R})^2\ dv_t
\\
&\leq\left(-\frac{Y}{3}+2\sqrt{6}\|W\|_2\right)\|W\|_4^2\notag
\\
&\quad+\left(-\frac{Y}{12}+(2+\frac{\sqrt{6}}{3})\|W\|_2+2\|E\|_2+(\frac{1}{12}+4a)\|R-\ol{R}\|_2\right)\|E\|_4^2\notag
\\
&\quad+\left(-\frac{Y}{3}a+\frac{2}{\sqrt{3}}a\|E\|_2\right)\|R-\ol{R}\|_4^2.\notag
\end{align}
Hence we see that $G_2$ is monotonically decreasing along the Ricci flow if
\begin{align}
-\frac{Y}{12}+2\sqrt{6 G_2}<0,\notag
\end{align}
or equivalently by the lower bound on $Y=Y(S^4,[g(t)])$ from Proposition \ref{cons1} \eqref{Ybound}, if
\begin{align}
F_2(0)=G_2(0)<\frac{16}{145}\pi^2.\notag
\end{align}

\subsection{The case $\delta>0$}

We use the same ideas, but now work with the quantities
\begin{align}
F_{2+\delta}(t)&=\int\|W\|^{2+\delta}+\frac{1}{2}|E|^{2+\delta}\ dv_t,
\\
G_{2+\delta}(t)&=\int \|W\|^{2+\delta}+\frac{1}{2}|E|^{2+\delta}+a|R-\ol{R}|^{2+\delta}\ dv_t,
\end{align}
for $\delta>0$ and $a$ to be specified later (recall from above that for $\delta=0$ we required $a\leq\frac{1}{192}$).

We first have the corresponding evolution formulas derived from Proposition \ref{evoform},
\begin{align}
\frac{\partial}{\partial t}||W||^{2+\delta}&\leq \Delta||W||^{2+\delta}-\frac{4\delta+4}{\delta+2}|\nabla||W||^{1+\delta/2}|^2
\\
&\quad+\left(1+\frac{\delta}{2}\right)||W||^{\delta}\left(36\det W^++36\det W^-+WEE\right)\notag
\\
\frac{\partial}{\partial t}|E|^{2+\delta}&\leq\Delta |E|^{2+\delta}-\frac{4\delta+4}{2+\delta}|\nabla|E|^{1+\delta/2}|^2
\\
&\quad+\left(1+\frac{\delta}{2}\right)|E|^{\delta}\left(4WEE-4\text{tr}E^3+\frac{2}{3}R|E|^2\right)\notag
\\
\frac{\partial}{\partial t}|R-\ol{R}|^{2+\delta}&\leq \Delta |R-\ol{R}|^{2+\delta}-\frac{4\delta+4}{\delta+2}|\nabla|R-\ol{R}|^{1+\delta/2}|^2\label{R2devopoint}
\\
&\quad+\left(2+\delta\right)|R-\ol{R}|^{1+\delta}\left(2|E|^2-2\ol{|E|^2}+\frac{1}{2}R^2-\ol{R}^2+\frac{1}{2}\ol{R^2}\right).\notag
\end{align}
Then we can compute to find that when $0\leq \delta<\sqrt{2}-1$,
\begin{align}
&\frac{d}{dt}\int\|W\|^{2+\delta}+\frac{1}{2}|E|^{2+\delta}\ dv_t\label{WEdelta}
\\
&\leq \left(-\frac{2(1+\delta)}{3(2+\delta)}Y+\left(1+\frac{\delta}{2}\right)\left(2\sqrt{6}+\frac{\sqrt{6}(1+\delta)}{3(3+\delta)}\right)\|W\|_2\right)\|W\|_{4+2\delta}^{2+\delta}\notag
\\
&+\left(-\frac{1+\delta}{6(2+\delta)}Y+2\left(1+\frac{\delta}{2}\right)\|W\|_2+\left(1+\frac{\delta}{2}\right)\left(2+\frac{\sqrt{6}}{3}\cdot\frac{2}{3+\delta}\right)\|E\|_2\right)\|E\|^{2+\delta}_{4+2\delta}.\notag
\end{align}
From this we see that $\frac{d}{dt}F_{2+\delta}(t)<0$ if $F_2(0)<\epsilon_0=\frac{16}{25}\pi^2$, so that $F_2(t)$ is nonincreasing, along with
\begin{align}
-\frac{1+\delta}{6(2+\delta)}Y+\left(1+\frac{\delta}{2}\right)\left(2\sqrt{2}+\frac{4}{\sqrt{3}(3+\delta)}\right)\sqrt{F_2(0)}<0,\notag
\end{align}
which, using Proposition \ref{cons1} \eqref{Ybound} is true whenever $F_2(0)$ is sufficiently small.
Indeed, since we require $\delta<\sqrt{2}-1$
the above inequality is true for instance whenever $F_2(0)<\frac{1}{5}\pi^2$.
We will in fact need a slightly more general inequality which we record below: if $F_2(0)<\frac{1}{5}\pi^2$ and $\delta<\sqrt{2}-1$, then
\begin{align}
\frac{d}{dt}F_{2+\delta}(t)+c_1 F_{4+2\delta}(t)^{\frac{1}{2}}&\leq- c_2\int R\|W\|^{2+\delta}\ dv_t-c_3\int R|E|^{2+\delta}\ dv_t
\\
&<0,\notag
\end{align}
where $c_1,c_2,c_3>0$ are constants, $c_1,c_2$ with positive lower bounds and $c_3$ tending to zero as $\delta\rightarrow\sqrt{2}-1$. 

We now need an estimate on the evolution of $\int |R-\ol{R}|^{2+\delta}\ dv_t$. By computing from \eqref{R2devopoint}, making use of the conformally invariant Sobolev inequality  \eqref{ysob}, and from Proposition \ref{R2boundp} that $\|R-\ol{R}\|_2\leq \sqrt{12}\|E\|_2$, we obtain

\begin{align}
\frac{d}{dt}\int|R-\ol{R}|^{2+\delta}\ dv_t&\leq \left(-\frac{2(1+\delta)}{3(2+\delta)} Y+\sqrt{12}\frac{(2+\delta)(5+3\delta)}{3+\delta}\|E\|_2\right)\|R-\ol{R}\|_{4+2\delta}^{2+\delta}\label{R2devo}
\\
&\quad+\left(\frac{2(1+\delta)}{3(2+\delta)}-1\right)\int R|R-\ol{R}|^{2+\delta}\ dv_t\notag
\\
&\quad+\frac{4(2+\delta)}{3+\delta}\|E\|_2\|E\|_{4+2\delta}^{2+\delta}\notag
\\
&\quad+(2+\delta)\ol{R}\int|R-\ol{R}|^{2+\delta}\ dv_t.\notag
\end{align}
Notice that the coefficient of the $\int R|R-\ol{R}|^{2+\delta}\ dv_t$ term is always negative. We can estimate the last term on the right by computing that
\begin{align}
\int|R-\ol{R}|^{2+\delta}\ dv_t&\leq \left(\int |R-\ol{R}|^2\ dv_t\right)^{\frac{2+\delta}{2+2\delta}}\left(\int|R-\ol{R}|^{4+2\delta}\ dv_t\right)^{\frac{\delta}{2+2\delta}}
\\
&\leq \eta^{-\frac{2+2\delta}{2}}\frac{24}{2+2\delta}\left(\int|E|^{2+\delta}\ dv_t\right)\notag
\\
&\quad+\eta^{\frac{2+2\delta}{2\delta}}\frac{2\delta\sqrt{12}}{2+2\delta}\vol(g)^{\frac{1}{2}}\left(\int|R-\ol{R}|^{4+2\delta}\ dv_t\right)^{\frac{1}{2}},\notag
\end{align}
for all $\eta>0$. Then, combining 
our estimates for
\eqref{WEdelta} and \eqref{R2devo} together, we find that
\begin{align}
&\frac{d}{dt}\int\|W\|^{2+\delta}+\frac{1}{2}|E|^{2+\delta}+a|R-\ol{R}|^{2+\delta}\ dv_t\label{G2devo}
\\
&\leq \left(-\frac{2(1+\delta)}{3(2+\delta)}Y+\left(1+\frac{\delta}{2}\right)\left(2\sqrt{6}+\frac{\sqrt{6}(1+\delta)}{3(3+\delta)}\right)\|W\|_2\right)\|W\|_{4+2\delta}^{2+\delta}\notag
\\
&+\left(-\frac{1+\delta}{6(2+\delta)}Y+2\left(1+\frac{\delta}{2}\right)\|W\|_2+\left(1+\frac{\delta}{2}\right)\left(2+\frac{\sqrt{6}}{3}\cdot\frac{2}{3+\delta}\right)\|E\|_2\right.\notag
\\
&\qquad\left.+\frac{\sqrt{12}}{2}\left|\frac{2(1+\delta)}{2+\delta} \cdot \frac{1}{6}+\frac{(2+\delta)}{3}-1\right|\|E\|_2+\frac{4a(2+\delta)}{3+\delta}\|E\|_2\right)\|E\|^{2+\delta}_{4+2\delta}.\notag
\\
&+a\left(-\frac{2(1+\delta)}{3(2+\delta)} Y+\sqrt{12}\frac{(2+\delta)(5+3\delta)}{3+\delta}\|E\|_2\right.\notag
\\
&\qquad\left.+\eta^{\frac{2+2\delta}{2\delta}}\frac{2\delta(2+\delta)\sqrt{12}}{2+2\delta}\vol(g)^{\frac{1}{2}}\ol{R}\right)\|R-\ol{R}\|_{4+2\delta}^{2+\delta}\notag
\\
&+\frac{1}{2}\left(\frac{2(1+\delta)}{2+\delta} \cdot \frac{1}{6}+\frac{(2+\delta)}{3}-1+2a\eta^{-\frac{2+2\delta}{2}}\frac{24(2+\delta)}{2+2\delta}\right)\ol{R}\int|E|^{2+\delta}\ dv_t.\notag
\end{align}
Above, note that the $\vol(g)^{\frac{1}{2}}\ol{R}$ term can be bounded from above in terms of $\|E\|_2$ by using \eqref{CGB}, since
\begin{align}
\ol{R}^2\leq \ol{R^2}\leq 24\vol(g)^{-1}\left(16\pi^2+\frac{1}{2}\|E\|_2^2\right).\notag
\end{align}
Therefore for fixed $\delta>0$ small, we can guarantee $\frac{d}{dt}\int\|W\|^{2+\delta}+\frac{1}{2}|E|^{2+\delta}+a|R-\ol{R}|^{2+\delta}\ dv_t<0$ as follows: in \eqref{G2devo} take $\eta>0$ sufficiently small so that the $\|R-\ol{R}\|_{4+2\delta}^{2+\delta}$ coefficient is negative for $G_2(0)$ sufficiently small, and then choose $a>0$ sufficiently small so that the $\ol{R}\int|E|^{2+\delta}\ dv_t$ coefficient is negative.  We can then, if needed, further restrict $G_2(0)$ to ensure that the remaining $\|W\|_{4+2\delta}^{2+\delta}$ and the $\|E\|_{4+2\delta}^{2+\delta}$ terms also have negative coefficients.

For example with some very rough estimates if we restrict $\delta\in[0,\frac{1}{3}]$, then taking $\eta=\frac{1}{100}$ and $a=10^{-6}$ will allow us to conclude that $\frac{d}{dt} G_{2+\delta}(t)<0$ whenever $F_2(0)<\frac{1}{1000}\pi^2$. For such a choice we in fact have that
\begin{align}
\frac{d}{dt}G_{2+\delta}(t)+C G_{4+2\delta}(t)^{\frac{1}{2}}\leq 0,\label{ddtineq}
\end{align}
for some $C>0$. We summarize this result below after converting to the normalized Ricci flow.

\begin{Prop}\label{final}
If $(S^4,g(t))$ is a normalized Ricci flow starting from a unit volume positive Yamabe metric $(S^4,g_0)$ with $\int\|W\|^2_{g_0}\ dv_{g_0}<\frac{1}{2000}\pi^2$, and we define
\begin{align}
G_p(t)=\int \|W\|^p+\frac{1}{2}|E|^p+\frac{1}{10^6}|R-\ol{R}|^p\ dv_t,
\end{align}
then for $p=2+\delta$ and any $\delta\in[0,\frac{1}{3}]$ there exists $C>0$ such that
\begin{align}
\frac{d}{dt}G_{2+\delta}(t)\leq -C G_{2+\delta}(t).
\end{align}
In particular, $G_{2+\delta}(t)$ is exponentially decreasing under the normalized Ricci flow (as long as it exists). 
\end{Prop}

\begin{Rem}
Although we have only established that $G_{2+\delta}$ is monotonically decreasing in Proposition \ref{final} above when $\int\|W\|^2\ dv_{g_0}$ is sufficiently small, standard Moser iteration arguments using Proposition \ref{usob} (similar to those used to derive Proposition \ref{DPW} below) will yield decay of $G_{2+\delta}$ --- namely, that there exist constants $C_1,C_2>0$ such that for all $t>0$ whenever the normalized flow exists,
\begin{align}
G_{2+\delta}(t)^{\frac{1}{2+\delta}}\leq C_1 t^{-\frac{1}{2+\delta}} e^{-C_2 t} G_2(0)^{\frac{1}{2}}.\label{2deltadecay}
\end{align}
In the paper of Gursky \cite{Gursky}, he actually established that there when $\int\|W\|^2\ dv_{g_0}$ is sufficiently small, \eqref{2deltadecay} holds for $\delta=1$ on the interval $[0,t_1]$, where $t_1>0$ is bounded from below, ie.
\begin{align}
G_3(t)^{\frac{1}{3}}\lesssim\frac{1}{t^{1/3}} G_2(0)^{\frac{1}{2}}.\label{Gurskydecay}
\end{align}
It is also possible to also establish the estimate \eqref{Gurskydecay} using \eqref{2deltadecay}, although one does not know whether $G_3(t)$ is also monotonically decreasing.
\end{Rem}

We have therefore proven the monotonicity and decay assertions in the main result of this Section, Theorem \ref{3main}. To finish this section we describe below how the remaining assertions on pointwise decay and convergence along the normalized Ricci flow can be obtained by standard Moser iteration arguments.

\subsection{Bounds in terms of the initial $L^2$-pinching}\label{linftysec}

We outline the argument for establishing $L^\infty$ bounds on $\|W\|$, $|E|$, and $|R-\ol{R}|$ along the normalized flow. First we recall from \eqref{ddtineq} that
\begin{align}
\frac{d}{dt}G_{2}(t)+C G_{4}(t)^{\frac{1}{2}}\leq 0,\notag
\intertext{and in fact for all $\delta\in[0,\frac{1}{3}]$,}
\frac{d}{dt}G_{2+\delta}(t)+C G_{4+\delta}(t)^{\frac{1}{2}}\leq 0.\notag
\end{align}
Next, for $0<\tau<\tau'<T$, multiply consider the piecewise continuous function $\psi:[0,T]\rightarrow[0,1]$,
\begin{align}
\psi(t)=
\begin{cases}
0,\quad&0\leq t\leq\tau
\\
\frac{t-\tau}{\tau'-\tau},\quad&\tau\leq t\leq\tau'
\\
1,\quad&\tau'\leq t.
\end{cases}\notag
\end{align}
Multiplying $\psi$ against the differential inequalities above and integrating, we get for $t_0\geq\tau'$ that
\begin{align}
G_2(t_0)+C\int_{\tau'}^{t_0}G_4(s)^{\frac{1}{2}}\ ds&\leq \frac{1}{\tau'-\tau}\int_\tau^{t_0}G_2(s)\ ds\label{p2}
\\
G_{2+\delta}(t_0)+C\int_{\tau'}^{t_0}G_{4+2\delta}(s)^{\frac{1}{2}}\ ds&\leq \frac{1}{\tau'-\tau}\int_\tau^{t_0}G_{2+\delta}(s)\ ds.\label{p3}
\end{align}
To continue from this point we no longer need to worry about the precise ratios between $\|W\|$, $|E|$, and $|R-\ol{R}|$ as in the definitions of $G_{2+\delta}$, so we let $K= |E|+ ||W||+|R-\ol{R}|$, and estimate

\begin{align}
\int K^{2+\delta}\ dv_t&\leq \frac{C}{t}\left(\sup_{s\in[\frac{t}{2},t]}\left(\int K^2\ dv_s\right)^{1-\frac{\delta}{2}}\right)\int_{\frac{t}{2}}^t\left(\int K^{4}\ dv_s\right)^{\frac{\delta}{2}}\ ds\label{penultdec}
\\
&\leq \frac{C}{t}\left(\sup_{s\in[\frac{t}{2},t]}\left(\int K^2\ dv_s\right)^{1-\frac{\delta}{2}}\right)\left(\int_{\frac{t}{2}}^t\left(\int K^{4}\ dv_s\right)^{\frac{1}{2}}\ ds\right)^\delta t^{1-\delta}\notag
\\
&\leq \frac{C}{t^{2\delta}}\left(\sup_{s\in[\frac{t}{2},t]}\left(\int K^2\ dv_s\right)^{1-\frac{\delta}{2}}\right)\left(\int_{\frac{t}{2}}^t\int K^2\ dv_s\ ds\right)^\delta\notag
\\
&\leq \frac{C}{t^\delta}||K||_2^{2+\delta}(0),\notag
\end{align}
in the last line using the monotonic decreasing property of $G_2$.  Hence
\begin{align}
\|K\|_{2+\delta}\leq C t^{-\frac{\delta}{2+\delta}}\|K\|_2(0).\label{Kdecay}
\end{align}

\begin{Rem}
Actually, we have a stronger estimate than above, since we actually have that $\|K\|_{2}^2(t)\leq \|K\|_2^2(0) e^{-Ct}$. So plugging this in to our estimation of \eqref{penultdec} we find
\begin{align}
\|K\|_{2+\delta}\leq \frac{C}{t^{\delta/(2+\delta)}}e^{-Ct}\|K\|_2(0),\label{qcontrol}
\end{align}
where here the constant $C$ may be different in different lines.
\end{Rem}


To complete our estimates, we will need Moser's weak maximum principle, and in partiucular, the following version:

\begin{Prop}[{\cite[Theorem 4]{Yang}}]\label{DPW}
	Let $f,b$ be smooth nonnegative functions satisfying on $M\times[0,T]$,
	\begin{align}
	\frac{\partial}{\partial t} f\leq\Delta f+bf,\notag
	\end{align}
	where $\Delta$ is the Laplace-Beltrami operator of the metric $g_t$, and suppose $\frac{\partial}{\partial t} dv_{g_t}=h_t dv_{g_t}$. Let $A,B>0$ be such that
	\begin{align}
	\|u\|_{\frac{2n}{n-2}}^2\leq A\|\nabla u\|_2^2+B\|u\|_2^2,\notag
	\end{align}
	for all $u\in C^\infty(M)$ and for all $t\in[0,T]$, and assume that for some $q>n/2$,
	\begin{align}
	\max_{0\leq t\leq T}(\|b\|_{q}+\|h_t\|_q)\leq\beta.\notag
	\end{align}
	Then given $p_0>1$, there exists a constant $C=C(n,q,p_0)$ such that for all $x\in M$ and $t\in(0,T]$,
	\begin{align}
	|f(x,t)|\leq C A^{\frac{n}{2p_0}}\left[\frac{B}{A}+A^{\frac{n}{2q-n}}\beta ^{\frac{2q}{2q-n}} +\frac{1}{t}\right]^{\frac{n+2}{2p_0}}\left(\int_0^t\int f^{p_0}\ dv_t\ dt\right)^{\frac{1}{p_0}}.\label{Yangestimate}
	\end{align}
\end{Prop}


First, we need to check that the normalized flow actually exists for all positive times. To see this, let $f=|\Rm|$. Under the Ricci flow we have $\frac{\partial}{\partial t}|\Rm|\leq \Delta|\Rm|+C|\Rm|^2$, so that this inequality also holds under the normalized Ricci flow, and moreover we have bounds on $\|\Rm\|_{2+\delta}$ for all positive times (if the flow exists) along the normalized flow by \eqref{qcontrol}. Since we also have the uniform Sobolev inequality along the normalized flow by Proposition \ref{usob}, then all of this information together substituted into Proposition \ref{DPW} shows that $|\Rm|$ is bounded away from infinity for all positive times along the normalized Ricci flow. Hence the normalized Ricci flow indeed exists for all positive times.




Next, we study the decay of $K$.
By \eqref{Wevo}, \eqref{Eevo}, \eqref{Revo}, we have
\begin{align}
\frac{\partial}{\partial t}K\leq\Delta K+CK^2+C\ol{R}K+\|E\|_2^2\label{Pineq}
\end{align}
But recall that $\overline{R}$ is bounded along the normalized flow so $C\ol{R}K\leq CK$. Recall also that $\|E\|_2^2\leq \|G\|_2^2(0) e^{-Ct}$. Therefore if we set
\begin{align}
\hat{K}=K+C^{-1}\|K\|_2^2(0) e^{-Ct},\notag
\end{align}
then we have
\begin{align}
\frac{\partial}{\partial t}\hat{K} \leq\Delta \hat{K}+C\hat{K}^2+C\hat{K}.\label{hateq}
\end{align}



Starting from \eqref{hateq}, we now estimate
\begin{align}
\|\hat{K}\|_{2+\delta}(t)\leq C'\|\hat{K}\|_{2+\delta}(0) e^{-Ct},\label{decKhat}
\end{align}
under the hypotheses of Proposition \ref{final}.
So we can apply Proposition \ref{DPW} to $\hat{K}$ to estimate $\|K\|_\infty$ along the flow.  We give the estimate in the form below with applications to related problems in mind:

\begin{Thm}
If $(S^4,g(t))$ is a normalized Ricci flow starting from a unit volume positive Yamabe metric $(S^4,g_0)$ with $\int\|W\|^2_{g_0}\ dv_{g_0}<\frac{1}{2000}\pi^2$, then for all $t>0$ and $p_0\in[2,2+\frac{1}{3}]$,
\begin{align}
&\|\|W\|+|E|+|R-\ol{R}|\|_\infty(t)
\leq C\left(1+t^{-\frac{2}{p_0}}\right) e^{-C' t}\|\|W\|+|E|\|_{p_0}(0),
\end{align}
where $C, C'>0$ are dimensional constants and independent of $g_0$. In particular, when $p_0=2$ we have
\begin{align}
&\|\|W\|+|E|+|R-\ol{R}|\|_\infty(t)
\leq C\left(1+t^{-1}\right) e^{-C' t}\|W\|_{2}(0).
\end{align}
\end{Thm}
\begin{proof}
Apply Proposition \ref{DPW} with $f=\hat{K}$ and $b=\hat{K}+C$, so that $\beta\leq \|\hat{K}\|_q+C$. Since we have a uniform Sobolev inequality along the flow, we may apply \eqref{Yangestimate} with a fixed choice of $q\in\left(2,2+\frac{1}{3}\right]$ (for instance, take $q=2+\frac{1}{3}$) and any $p_0\in\left[2,2+\frac{1}{3}\right]$ to find that
\begin{align}
\|\hat{K}\|_\infty(t)\leq C\left(1+\|\hat{K}\|_q^{\frac{2q}{2q-4}}+t^{-1}\right)^{\frac{3}{p_0}}t^{\frac{1}{p_0}}e^{-C' t}\|\hat{K}\|_{p_0}(0),\label{finalest}
\end{align}
where we have obtained the extra exponential decay factor by estimating $\|\hat{K}\|_\infty(t)$ by the bounds on $\|\hat{K}\|_{p_0}$ for times in $[t/2,t]$, taking advantage of \eqref{decKhat}. This also allows us to estimate by \eqref{Kdecay} that on $[t/2,t]$,
\begin{align}
\|\hat{K}\|_q^{\frac{2q}{2q-4}}&\leq \left(C+C t^{-\frac{q-2}{q}}\right)^{\frac{2q}{2q-4}}\notag
\\
&\leq C+Ct^{-1}.\notag
\end{align}
Putting this back into \eqref{finalest} and rewriting in terms of $K$ while using that $\|K\|_2(0)$ is bounded by a dimensional constant, the conclusion follows. 
\end{proof}

\subsection{An open question}

Throughout the arguments above we have made extensive use of the Chern--Gauss--Bonnet identity \eqref{CGB}, which is is one reason why all of the results described here are restricted to dimension four. A natural question is whether one can obtain sphere theorems based on conformally invariant hypotheses such as \eqref{Acond} in Theorem A of \cite{CGY03}, but in other dimensions. Recent work of Chen--Wei--Ye \cite{CWY20} can be viewed as a step in this direction:

\begin{Thm}[{\cite[Theorem 1.1]{CWY20}}]\label{CWY}
There exists $\Lambda(n)>0$ such that if $(M^n,g)$ is a compact manifold with a Yamabe metric $g_0\in[g]$ which satisfies
\begin{align}
\|W_{g_0}\|_{n/2}+\|E_{g_0}\|_{n/2}<\Lambda(n) Y(M,[g]),\label{CWYeq}
\end{align}
then $M^n$ is diffeomorphic to an isometric quotient of $S^n$. Moreover, the normalized Ricci flow starting from $(M^n,g_0)$ exists for all times and converges to a quotient of $(S^n,g_c)$.
\end{Thm}
Theorem \ref{CWY} is a sort of generalization to general dimensions of Theorem \ref{3main}, which holds when $n=4$.  Unfortunately, unlike the condition on $\int |W|^2\ dv$ in four dimensions, the condition \eqref{CWYeq} above, although scaling invariant, is not conformally invariant; rather, it depends on checking the inequality at a special choice of metric in $[g]$ (a Yamabe metric). It would be very interesting to see if \eqref{CWYeq} could be improved to a truly conformally invariant inequality, such as for instance by comparing the left-hand side for a Yamabe metric to the infimum over all metrics in $[g]$.

\bibliographystyle{alpha}
\bibliography{references}

\end{document}